\documentclass[a4paper,11pt,reqno]{amsart}
\usepackage{amsmath}
\usepackage[T1]{fontenc}
\usepackage{amssymb}
\usepackage{amsthm}
\usepackage{color}
\usepackage[lmargin=2.5 cm,rmargin=2.5 cm,tmargin=3.5cm,bmargin=2.5cm,
paper=a4paper]{geometry}

\usepackage{graphicx}

\newcommand{\Ab}{\mathbf A}

\newcommand{\Fb}{\mathbf F}

\newcommand{\R}{\mathbb R}
\newcommand{\C}{\mathbb C}
\newcommand{\E}{\mathrm{E}_{\rm gs}(\kappa, H)}

\DeclareMathOperator{\curl}{curl}

\newtheorem{thm}{Theorem}[section]
\newtheorem{prop}[thm]{Proposition}

\theoremstyle{remark}
\newtheorem{assump}[thm]{Assumption}
\newtheorem{rem}[thm]{Remark}

\numberwithin{equation}{section}

\title[Ginzburg-Landau density]
{The density of superconductivity in domains with corners}

\author{Bernard Helffer}
\author{Ayman Kachmar}
\address[B. Helffer]{Laboratoire Jean Leray, Universit\'e de Nantes, 2 rue de la Houssini\`ere, 44322 Nantes (France) and Laboratoire de Math\'ematiques, Univ. Paris-Sud. }
\email{bernard.helffer@univ-nantes.fr}
\address[A. Kachmar]{Department of Mathematics, Lebanese University, Nabatieh, Lebanon.}
\email{ayman.kashmar@gmail.com}
\date{\today}

\begin{document}

\begin{abstract}
We compute the $L^2$-norm of the minimizer of the Ginzburg-Landau functional in a  planar domain with a finite number of corners.  Our computations are valid  for a uniform applied magnetic field,  large Ginzburg-Landau parameter and in the regime where superconductivity is confined near the corners of the domain.
\end{abstract}

\maketitle 

\section{Introduction and main results}\label{hc2-sec:int}

\subsection{Superconductivity in domains with corners}
Superconducting samples with non-smooth cross sections are interesting for their particular response to the applied magnetic field. When a superconducting sample is submitted to a uniform applied magnetic field of very large intensity, superconductivity breaks down and the  sample returns to the normal conducting state (see \cite{GP}). When the intensity of the magnetic field is decreased just below a certain threshold value $H_{C_3}$, superconductivity appears again on the surface of the sample. The value $H_{C_3}$ is called the third critical field and is computed via a linear spectral problem (see \cite{FH}). The existing results (in physics and mathematics) show that $H_{C_3}$ depends on the geometry of the sample's cross section. In particular, 
the value of $H_{C_3}$ is significantly larger when the cross section of the sample has corners than for samples of the same material but with smooth cross section. That has been early observed in the Physics literature \cite{BrDFM}, then established by  rigorous analysis of the Ginzburg-Landau functional.  We refer  to the works of Pan \cite{Pa}, Pan-Kwek \cite{KwPa}, 
Jadallah \cite{Ja}, Bonnaillie \cite{Bon}, Bonnaillie-No\"el--Dauge \cite{BD},  Bonnaillie-No\"el--Fournais  \cite{BF},  and to Chapter 16 in the book \cite{FH} for the state of the art in 2014.

The experimentally observed change in the value of the third critical $H_{C_3}$ is due to the fact that  the first
eigenvalue for the magnetic Neumann Laplacian is asymptotically smaller when a domain has a
corner. Eigenfunctions corresponding to the lowest eigenvalues will be
localized near the corners and their leading order large field asymptotics 
 are controlled by the model of an infinite sector \cite{Bon,BD}.

Recently,  Correggi-Giacomelli \cite{CG}  studied the non-linear aspects of surface superconductivity in domains with corners,   and Exner-Lotoreichik-P\'erez-Obiol \cite{ELP} obtained new estimates on the spectral model in an infinite sector. In this contribution, we implement the recent improvements for the analysis of bulk superconductivity in \cite{HK} in order  to provide a precise description of the confinement of superconductivity near the corners of the domain. Our results sharpen the results by Bonnaillie-No\"el--Fournais \cite{BF} and are complementary to those of Correggi-Giacomelli \cite{CG}. The later article is devoted to the regime where superconductivity is uniformly distributed along the whole surface of the sample.

\subsection{Domains with corners}

Let $\Omega\subset\R^2$ be an open, bounded and simply connected domain. Assume that the boundary $\Gamma$ of the domain $\Omega$ is a curvilinear polygon of class $C^3$ (see \cite[p.~34-42]{G}). By this we mean that for all $x\in\Gamma$, there exists a neighborhood $V_x$ of $x$ in $\R^2$ and an injective mapping $\psi^x=(\psi^x_1,\psi^x_2):V_x\to\R^2$ such that:  
\begin{itemize}
\item $\psi^x$ and $(\psi^x)^{-1}$ are of class $C^3$\,;
\item $\Omega\cap V_x$ is either $\{y\in\Omega~:~\psi^x_2(y)<0\}$, $\{y\in\Omega~:~\psi^x_1(y)<0~\&~\psi^x_2(y)<0\}$ or $\{y\in\Omega~:~\exists~j\in\{1,2\},~ \psi^x_j(y)<0\}$.
\end{itemize}
When $\Omega\cap V_x=\{y\in\Omega~:~\psi^x_2(y)<0\}$, the point $x$ is said to be smooth. Otherwise,  $x$ is said to be a corner point.
We assume that $\Gamma$ consists exactly of $m\geq 2$ connected simple smooth curves $(\Gamma_k)_{k=1}^m$ such that, 
\begin{equation}\label{eq:G**}
\Gamma_k\cap\Gamma_{k'}=\emptyset\quad{\rm for}~(k,k')\in\{(i,j)\not=(m,1)~:~i-j>1\}\,,
\end{equation}
and
\begin{equation}\label{eq:G}
\forall~k\in\{1,\cdots,m\},\quad \Gamma_{k-1}\cap\Gamma_{k}=\{\mathsf s_k\}\,,
\end{equation} 
with  the convention $\Gamma_0:=\Gamma_m$.  The number $m$ is assumed to be  the minimal number such that the boundary consists of $m$ smooth curves. The points $\mathsf s_k$, $k\in\{1,\cdots,m\}$, are the vertices (corners) of  the domain $\Omega$. Furthermore, for $1\leq k\leq m$, we assume that the curve $\Gamma_k$ is oriented counter clockwise from $\mathsf s_k$ to $\mathsf s_{k+1}$, and we denote by $\alpha_k$ the (interior) angle at $\mathsf s_k$ between $\Gamma_{k-1}$ and $\Gamma_k$ (see Figure~1). In the sequel, 
\begin{equation}\label{eq:Sigma}
\Sigma=\{\mathsf s_1,\cdots,\mathsf s_m\}\,.
\end{equation}
\begin{figure}
\center
\includegraphics[scale=3.5]{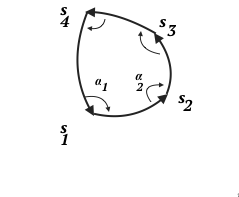}
\caption{The domain $\Omega$ with its boundary $\Gamma$ and the vertices $s_1,s_2,\cdots$.}
\end{figure}
\subsection{The Ginzburg-Landau model}

Assume that $\Omega$ is the horizontal cross section of a superconducting wire submitted to the applied magnetic field $\kappa H \vec{e}$, where $\vec{e}=(0,0,1)$, $\kappa$ is the Ginzburg-Landau parameter and $H>0$ measures the intensity of the magnetic field. The superconducting properties of the sample are described by a configuration $(\psi,\Ab)\in\mathcal H$ minimizing the Ginzburg-Landau energy
\begin{equation}\label{eq-3D-GLf}
\mathcal E_{\rm GL}(\psi,\Ab)=
\int_\Omega \left(|(\nabla-i\kappa H\Ab)\psi|^2-\kappa^2|\psi|^2+\frac{\kappa^2}2|\psi|^4\right)\,dx+(\kappa H)^2\int_{\R^2}|\curl\Ab-1|^2\,dx\,.
\end{equation}
The  space $\mathcal H$ is   defined as follows  (see Appendix D in \cite{FH})
\begin{equation}\label{eq:space}
\mathcal H=\{(\psi,\Ab)~:~\psi\in W^{1,2}(\Omega;\C)~\&~\Ab-\Fb\in  W^{1,2}_{0,0}(\R^2;\R^2)\}
\end{equation}
where 
\begin{equation}\label{eq:space*}
W^{1,2}_{0,0}(\R^2;\R^2)=\{\mathbf u\in W^{1,2}_{\rm loc}(\mathbb R^2;\R^2)~:~\frac{\mathbf u}{\sqrt{1+x^2}\,\ln(2+x^2)}\in L^2(\R^2;\R^2)\,,\nabla \mathbf u\in L^2(\R^2;\R^2))\}
\end{equation}
and
\begin{equation}\label{eq:F}
\Fb(x_1,x_2)=\frac12 (-x_2,x_1)\,.
\end{equation}
The ground state energy of the functional in
\eqref{eq-3D-GLf} is:
\begin{equation}\label{eq-gse}
\E=\inf\{\mathcal E_{\rm GL}(\psi,\Ab)~:~(\psi,\Ab)\in\mathcal H\}\,.
\end{equation}
We consider  $(\psi,\Ab) \in \mathcal H$ a minimizer of the functional $\mathcal E_{\rm GL}$. We   will denote it by $(\psi,\Ab)_{\kappa,H}$ to emphasize its dependence on $\kappa$ and $H$. 

A minimizer $(\psi,\Ab)_{\kappa,H}$  is a solution of the following Ginzburg-Landau equations (we use the notation $\nabla^\bot=(\partial_{x_2},-\partial_{x_1})$)
\begin{equation}\label{eq:GL}
\left\{
\begin{array}{rll}
-\big(\nabla-i\kappa H\Ab\big)^2\psi&=\kappa^2(1-|\psi|^2)\psi &{\rm in}\ \Omega\,,\\
-\nabla^{\perp}  \curl\Ab&= (\kappa H)^{-1}{\rm Im}\big(\overline{\psi}\,(\nabla-i\kappa H {\bf A})\psi\big) & {\rm in}\ \Omega\,,\\
\nu\cdot(\nabla-i\kappa H\Ab)\psi&=0 & {\rm on}\ \partial \Omega \,,\\
\curl \Ab &=B_0 & {\rm in}\ \R^2\setminus\overline{\Omega} \,.
\end{array}
\right.
\end{equation} 
The following quantities 
\begin{equation}
|\psi|^2\,,\quad\curl\Ab\,,\quad |(\nabla-i\kappa H\Ab)\psi|^2\,,\,\mbox{ and } 
j(\psi,\Ab)={\rm Re}\big(-i\overline{\psi}\,(\nabla-i\kappa H\Ab)\psi\big) 
\end{equation}
are invariant under the {\it gauge} transformation $(\psi,\Ab)\mapsto (e^{i\chi},\Ab-\nabla\chi)$, for any $\chi\in H^1(\Omega;\R)$. These are the  physically relevant quantities  which are called  {\it density}, {\it induced magnetic field}, {\it  kinetic energy density} 
and {\it supercurrent} respectively.

\subsection{The reference energy}

The behavior of the minimizers of the functional in \eqref{eq-3D-GLf} has been analyzed by Bonnaillie-No\"el and Fournais  in \cite{BF} (and recently by Correggi-Giacomelli in \cite{CG}). The focus in \cite{BF} was mainly on the asymptotics of the ground state energy in the regime where the minimizing order parameter is concentrated near the corners of the domain (\cite{CG} is devoted to the full surface regime). In that respect, a central role is played by the reference functional defined as follows:
\begin{equation}\label{eq:ref-en}
\forall~u\in W^{1,2}_{\Fb}(\Omega_\alpha)\,, \quad J_{\mu,\alpha}(u)=\int_{\Omega_\alpha}\left(|(\nabla-i\Fb)u|^2-\mu|u|^2+\frac{\mu}2|u|^4\right)\,dx\,,
\end{equation}
where $\mu>0$ and $\alpha\in (0,2\pi)$ are given, 
\begin{equation}\label{eq:Om-alpha}
\Omega_\alpha=\{(x_1,x_2)\in\R^2~:~x_1>0~\&~x_2<x_1 \tan\frac\alpha2\,\}
\end{equation}
is the infinite sector in $\R^2$ of opening $\alpha$, and
\begin{equation}\label{eq:sob-mag}
W^{1,2}_{\Fb}(\Omega_\alpha)=\{u\in L^2(\Omega_\alpha)~:~(\nabla-i\Fb)u\in L^2(\Omega_\alpha)\}\,.
\end{equation}
The functional in \eqref{eq:ref-en} is bounded from below if and only if $\mu\leq \Theta_0$, where $\Theta_0$ is a value defined  in \eqref{eq:Theta0} below.  This is a consequence of the spectral theory of the magnetic Schr\"odinger operator in the half-plane and the plane respectively.  In fact, if $K\Subset\Omega_\alpha$ is an unbounded domain such that $0\not\in \overline{K}$, and if $u\in  C_0^\infty(\R^2)$ with support in $\overline{K}$, then the following inequality holds
$$\int_{\Omega_\alpha}|(\nabla-i\Fb)u|^2\,dx\geq \int_{\Omega_\alpha}|u|^2\,dx\,.
$$
 If  $\overline{K}$ intersects one edge of $\Omega_\alpha$ and $u=0$ on $\Omega_\alpha\cap\partial K$, then we have the alternative inequality
$$\int_{\Omega_\alpha}|(\nabla-i\Fb)u|^2\,dx\geq \Theta_0\int_{\Omega_\alpha}|u|^2\,dx\,.
$$ 
%
%
Define the ground state energy 
\begin{equation}\label{eq:ref-en*}
E(\mu,\alpha)=\inf_{u\in W^{1,2}_{\Fb}(\Omega_\alpha)} J_{\mu,\alpha}(u)\,.
\end{equation}
The functional \eqref{eq:ref-en} has non-trivial minimizers if and only if $\mu$ satisfies a spectral condition, namely
\begin{equation}\label{eq:mu>mu1}
\mu>\mu_1(\alpha)\,,
\end{equation}
where $\mu_1(\alpha)$ is the smallest eigenvalue of the operator $P_\alpha=-(\nabla-i\Fb)^2$ in $L^2(\Omega_\alpha)$ with Neumann boundary conditions,
\begin{equation}\label{eq:mu1}
\mu_1(\alpha)=\inf_{u\in { W_\Fb^{1,2}(\Omega_\alpha)\setminus\{0\}}} \frac{\int_{\Omega_\alpha}|(\nabla -i\Fb)u|^2\,dx}{\int_{\Omega_\alpha}|u|^2\,dx}\,.
\end{equation}

The bottom of the essential spectrum of the operator $P_\alpha$ is independent of $\alpha$ (see \cite{Bon}) 
\begin{equation}\label{eq:Theta0}
\Theta_0=\inf\sigma_{\rm ess}(P_\alpha)
\end{equation}
and its approximate value is $0.59$. $\Theta_0$ can be defined starting from a family of harmonic oscillators on the semi-axis and is commonly called the de\,Gennes constant.  It is also the bottom of the spectrum of the operator $P_{\Omega_\alpha}$ in the half-space case, i.e. when $\alpha=\pi$.

\subsection{Earlier results}

In a specific asymptotic regime of the applied magnetic field,  superconductivity is confined  near at least one corner of the domain; see \cite{BF}. This happens under the following assumption on the domain:
\begin{assump}\label{ass:corners}
$$ \forall~\mathsf s_k \in\Sigma\,,\quad \mu_1(\alpha_k)<\Theta_0\,,$$
where $\Sigma$ is the set of vertices of the domain $\Omega$ introduced in \eqref{eq:Sigma}.
\end{assump}

Examples of angles $\alpha$ satisfying the spectral condition $\mu_1(\alpha)<\Theta_0$ are discussed in \cite{Bon, Pa}. Recently,  in \cite{ELP},
 it is proved that the inequality $\mu_1(\alpha)<\Theta_0$ holds for all $\alpha\in(0,0.595\pi)$.   It is actually conjectured in  \cite{Bon} (and supported by numerical evidence)  that this inequality holds if and only if $\alpha\in(0,\pi)$. 
If this conjecture was proved, Assumption \ref{ass:corners} would be simply  interpreted as the domain   $\Omega$ being convex.\\

If Assumption~\ref{ass:corners} holds, then for all $\mu\in (0,\Theta_0)$,  it is proven in \cite{BF} that the ground state energy in \eqref{eq-gse}  satisfies  for $H=\mu^{-1}\kappa$ and  $\kappa\to+\infty$
\begin{equation}\label{eq:BonF}
\E=\sum_{\mu_1(\alpha_k)<\mu}E(\mu,\alpha_k)+o(1)\,.
\end{equation}
This result suggests that superconductivity in this regime is confined in the corners satisfying the spectral condition $\mu_1(\alpha_k)<\mu\,$.  This  will be made precise by our results below.  

\subsection{New results}
 Under Assumption~\ref{ass:corners} and for
$$
\mu \in (0,\Theta_0)\,,\, H =\mu^{-1} \kappa\,,
$$
(hence with  \eqref{eq:BonF} satisfied), we will give in the limit $\kappa \rightarrow +\infty$,  for any minimizer $(\psi,\Ab)_{\kappa,H}$:
\begin{itemize}
\item the asymptotics for the local energy   (i.e. energy in any domain $D\subset\Omega$) of $\psi$,
\begin{equation}\label{eq:loc-en}
\mathcal E_{\rm GL}(\psi,\Fb;D):=\int_D \left(|(\nabla-i\kappa H\Fb)u|^2-\kappa^2|\psi|^2+\frac{\kappa^2}{2}|\psi|^4\right)\,dx\,,
\end{equation}
\item the asymptotics for the  local integrated density:
$$\int_D|\psi (x)|^2\,dx\quad{\rm and}\quad \int_D|\psi(x)|^4\,dx\,,$$
for an arbitrary open set $D$ in  $\Omega$.
\end{itemize}
Let us come back to the reference energy $E(\mu,\alpha)$ in \eqref{eq:ref-en*}. For every $\alpha$, it is easy to check that the function $\mu\mapsto E(\mu,\alpha)$ is concave. Consequently, the left- and right-sided derivatives with respect to $\mu$
$$E'(\mu_\pm,\alpha):=\lim_{\epsilon\to0_\pm}\frac{E(\mu+\epsilon)-E(\mu)}{\epsilon}$$
exist and the set
\begin{equation}\label{eq:E-diff}
\mathcal S_\alpha=\{\mu>0~:~E'(\mu_+,\alpha)\not=E'(\mu_-,\alpha)\}
\end{equation}
is at most countable.\\

For all $x\in\R^2$ and $\ell>0$, we use the notation
\begin{equation}\label{eq:N(x,ell)}
\mathcal N(x,\ell)=\{y\in\Omega~:~|y-x|<\ell\}\,.
\end{equation}
Our main result is:
\begin{thm}\label{thm:HK}

Suppose that Assumption~\ref{ass:corners} holds. For $\delta\in(\frac{4}5,1)$ and $\mu\in(0,\Theta_0)$, there exist $\kappa_0>0$ and
  a function $\lambda:\R_+\to\R_+$ such that $\displaystyle\lim_{\kappa\to+\infty}\lambda(\kappa)=0$ and the following holds true.

If $(\psi,\Ab)_{\kappa, H}$ is a minimizer of the functional in \eqref{eq-3D-GLf} for $H=\mu^{-1}\kappa$ and $\kappa\geq \kappa_0$, then for all  $k\in\{1,\cdots,m\}$, and with $\ell=\kappa^{-\delta}$, we have:
\begin{enumerate}
\item 
$$-E'(\mu_-,\alpha_k)-\lambda(\kappa)\leq \int_{\mathcal N(\mathsf s_k,\ell)}|(\nabla-i\kappa H\Ab)\psi|^2\,dx\leq -E'(\mu_+,\alpha_k)+\lambda(\kappa)\,.$$
\item 
$$-E'(\mu_-,\alpha_k)-E(\mu,\alpha_k) -\lambda(\kappa)\leq \kappa^2\int_{\mathcal N(\mathsf s_k,\ell)} |\psi(x)|^2\,dx\leq  -E'(\mu_+,\alpha_k)-E(\mu,\alpha_k)+ \lambda(\kappa)\,.$$
\item 
$$\left|\kappa^2\int_{\mathcal N(\mathsf s_k,\ell)} |\psi(x)|^4\,dx+2E(\mu,\alpha_k)\right|\leq \lambda(\kappa)\,.$$
\end{enumerate}
\end{thm}
\begin{rem}\label{rem:thm:HK}
Theorem~\ref{thm:HK} asserts that superconductivity is present in the set $\mathcal N(\mathsf s_k,\ell)$ as long as $\mu_1(\alpha_k)<\mu$, the necessary and sufficient condition to get $E(\mu,\alpha_k)\not=0$. 
In the particular case of a domain with symmetries, where   one finds two corners $\mathsf s_k$ and $\mathsf s_{k'}$ with the same angle $\alpha_k=\alpha_{k'}=:\alpha$, Theorem~\ref{thm:HK} demonstrates that superconductivity is present near both corners with the same strength   provided that $\mu_1(\alpha)<\mu$. This phenomenon can be seen as a non-linear form of the  tunneling  effect, expected to occur for the linear problem in domains with symmetries. We refer to \cite{BHR} for details reagrding the linear problem.
\end{rem}

\begin{rem}\label{rem:thm-cf}{\bf(Critical fields)} 
 By relabeling the vertices $\{\mathsf s_1,\cdots,\mathsf s_m\}$, we may assume that the corresponding angles satisfy
$\mu_1(\alpha_1)\geq \cdots\geq \mu_1(\alpha_m)$.
Combining the results in Theorem~\ref{thm:HK} above and those in \cite{BF, CG}, we may introduce  under Assumption \ref{ass:corners} the following critical fields,
$$
H_{C_2}\leq H_{C_3}^{0}< H_{C_3}^{1}\leq \cdots \leq H_{C_3}^m
$$
defined as follows
$$
H_{C_2}=\kappa\,,\quad H_{C_3}^{0}=\frac{\kappa}{\Theta_0}\,,\quad H_{C_3}^\ell=\frac{\kappa}{\mu_1(\alpha_{ \ell})}\quad(\ell\in\{1,\cdots,m\})\,,$$
Loosely speaking, these critical fields correspond to the following transitions. Below $H_{C_2}$, bulk superconductivity persists \cite{SS02}; between $H_{C_2}$ and $H_{C_3}^{0}$, surface superconductivity is uniformly distributed (in the $L^2$ sense) along the whole surface \cite{CG};  between $H_{C_3}^{\ell}$ and $H_{C_3}^{\ell+1}$ (with $0\leq \ell \leq m-1$ and $\mu_1 (\alpha_{\ell }) > \mu_1 (\alpha_{\ell+1}) )$, superconductivity is confined near the vertex $\mathsf s_k$ (\cite{BF} and Theorem~\ref{thm:HK} above); above $H_{C_3}^m$, superconductivity disappears everywhere \cite{BF}.  The critical field $H_{C_3}^m$ is actually the principal term of the third critical field $H_{C_3}$ whose precise definition and asymptotics are discussed in  \cite{BF}.
\end{rem}
\section{Former estimates}

\subsection{On the global Ginzburg-Landau energy}
We recall  that if $(\psi,\Ab)_{\kappa,H}$ is a critical point to the functional in \eqref{eq-3D-GLf}, then we can apply 
a gauge transformation and assume that $(\psi,\Ab)_{\kappa, H}$ satisfies (cf. \cite[pp. 248-249, Eqs.~(15.17)-(15.19)]{FH} and \cite{BF})
\begin{equation}\label{eq:mag-en}
{\rm div}\Ab=0~{\rm in~}\Omega\quad{\rm and}
\quad \|\Ab-\Fb\|_{W^{1,2}(\Omega)}\leq C\, \|\curl\Ab-1\|_{L^2(\R^2)}\,.
\end{equation}
Hereafter,  a critical point of the functional in \eqref{eq-3D-GLf} is supposed to satisfy \eqref{eq:mag-en}. We will not mention this explicitly afterwards.

Next  we collect useful estimates regarding the critical points of the Ginzburg-Landau functional 
(cf. \cite[Eqs.~(15.20)-(15.22) \& Lem.~15.3.3]{FH}).

\begin{prop}\label{prop:FH-b}
Let $\mu\in(0,\Theta_0)$.  There exist two constants
$C>0$ and $\kappa_0>0$ such that, if $\kappa\geq \kappa_0$, $H=\mu^{-1}\kappa$ and $(\psi,\Ab)_{\kappa,H}$ is a critical point of the functional in \eqref{eq-3D-GLf}, then:
\begin{align}
&\|\psi\|_{L^\infty(\Omega)}\leq1\,,\label{eq:psi<1}\\
&\|(\nabla-i\kappa H\Ab)\psi\|_{L^2(\Omega)}\leq C\, \kappa \, \|\psi\|_{L^2(\Omega)}\,,\label{eq:grad-psi}\\
& \|\curl\Ab-1\|_{L^2(\R^2)}\leq \frac{C}{\kappa}\, \|\psi\|_{L^2(\Omega)}\, \|\psi\|_{L^4(\Omega)}\,,\label{eq:curl=cst}\\
&\|\Ab-\Fb\|_{L^4(\Omega)}\leq \frac{C}{\kappa}\, \|\psi\|_{L^2(\Omega)}\,\|\psi\|_{L^4(\Omega)}\,.\label{eq:A=F}
\end{align}
\end{prop}

To state the results concerning the  concentration of the order parameter near the corners $(\mathsf s_k)_{1\leq k\leq n}\,$, we introduced the following set of vertices, for $\mu>0\,$,
\begin{equation}\label{eq:vertices*}
\Sigma'(\mu)=\{\mathsf s_k~:~\mu_1(\alpha_k)\leq \mu\}\,.
\end{equation}
We also introduce the following quantity
\begin{equation}\label{eq:Lambda1}
\Lambda_1=\min_{ 1\leq k\leq m}\mu_1(\alpha_k)\,.
\end{equation} 
 
\begin{prop}\label{prop:BonF} {\rm  (\cite{BF})}

Assume that $\Lambda_1$ satisfies $0<\Lambda_1<\Theta_0\,$. Given $\mu\in(\Lambda_1,\Theta_0)$, there exist positive constants $\kappa_0$, $\epsilon$ and $C$ such that, for all $\kappa\geq \kappa_0$, $H=\mu^{-1}\kappa$ and $(\psi,\Ab)_{\kappa,H}$  a solution of \eqref{eq:GL}, 
$$\int_{\Omega}e^{\epsilon\,\kappa{\rm \,dist}(x,\Sigma'(\mu))} \left(|\psi|^2+\frac{1}{\kappa^2}|(\nabla-i\kappa H\Ab)\psi|^2\right) \,dx\leq \frac{C}{\kappa^2}\,.$$
\end{prop} 

\subsection{Change of coordinates}

Let $\mathsf s=\mathsf s_k\in\Sigma$ be a vertex of the $\Omega$ (see \eqref{eq:Sigma}) and $\alpha=\alpha_k$. 
We may select a constant $\ell_0>0$ (independent of $\mathsf s$) and a diffeomorphism  $\Phi_{\mathsf s}$ of $\R^2$ such that $\Phi_{\mathsf s}=0$ and, and for all $\ell\in(0,\ell_0)$, $\Phi_{\mathsf s}(\mathcal  N(\mathsf s,\ell))=\Omega_\alpha\cap \Phi_{\mathsf s}(B(\mathsf s,\ell))$ and $|{\rm det}(D\Phi_{\mathsf s})(\mathsf s)|=1$, where $\mathcal N(\mathsf s,\ell)$ and $\Omega_\alpha$ are introduced in \eqref{eq:N(x,ell)} and \eqref{eq:Om-alpha} respectively. This transformation was used previously in \cite[Sec.~6.2]{BF}.

For every $(u,A)$ defined on $\mathcal N(\mathsf s,\ell)$, we assign the configuration $(\tilde u,\tilde A)$ by the change of coordinates $x\mapsto y= \Phi_{\mathsf s}(x)$. Assuming that $ {\rm supp}\,u\subset\{|x-\mathsf s|<\ell\}$,  we get the change of variables formulas
$$
\int_{\mathcal N(\mathsf s,\ell)}|(\nabla-iA)u|^2\,dx=\int_{\Omega_\alpha} \langle (\nabla-i\tilde A)\tilde u\,,\,G(y)(\nabla-i \tilde A)\tilde u\rangle\, a(y)\,dy\,,
$$
and
$$\int_{\mathcal N(\mathsf s,\ell)}|u(x)|^2\,dx=\int_{\Omega_\alpha} |\tilde u(y)|^2 a(y)\,dy\,,
$$
where $a(y)=|{\rm det}(D\Phi_{\mathsf s}^{-1})(y)|$, $G(y)=(D\Phi_{\mathsf s})(D\Phi_{\mathsf s})^T\big|_{\Phi_{\mathsf s}(y)}$,  and $ \tilde u (y)=u\big(\Phi_{\mathsf s}^{-1}(y)\big)$.

Since $$A_1\,dx_1+A_2\,dx_2=\tilde A_1\, dy_1+\tilde A_2\, dy_2\,,$$
 $\tilde A$ generates the magnetic field
$$\tilde B(y)=\partial_{y_1}\tilde A_2-\partial_{y_2}\tilde A_1= a(y) B(\Phi_{s}^{-1}(y))\,,$$
where $$B=\curl A=\partial_{x_1}A_2-\partial_{x_2}A_1\,.$$

In the particular case when the magnetic potential  is $\Fb$,  we can pick a smooth function $\eta$ such that
\begin{equation}\label{eq:F-eta}
\tilde\Fb(y)=\hat\Fb(y)+\nabla\eta(y)
\end{equation} 
where
$$
\hat\Fb(y)=\left(-\frac{y_2}2\,,\,\int_{0}^{y_1} \left(a(t,y_2)-\frac12\right)\,dt\right)=\Fb(y)+\mathcal O(\ell^2)\quad{\rm in}~\Omega_\alpha\cap \{|y|<\ell\}\,.
$$

\subsection{On the reference energy}~\\
We now come back to the reference functional introduced  in \eqref{eq:ref-en}.  Assume that $\alpha\in (0,2\pi)$,   $\mu_1(\alpha)<\Theta_0$\,,
and  $\mu\in(\mu_1(\alpha),\Theta_0)$. Under these conditions,  the functional in \eqref{eq:ref-en} is bounded from below and has a non identically $0$  minimizer $u_\alpha(=u_{\alpha,\mu})$ satisfying 
(cf. \cite{BF} or \cite[Prop.~15.3.10]{FH})
\begin{equation}\label{eq:ref-min<1}
\|u_\alpha\|_{L^\infty(\Omega_\alpha)}\leq 1\,,
\end{equation}
and
\begin{equation}\label{eq:ref-minimizer}
\int_{\Omega_\alpha} e^{2a_0 |x|} \big(|u_\alpha(x)|^2+|(\nabla-i\Fb)u_\alpha (x)|^2\big)\,dx\leq C\,,
\end{equation}
where $a_0$ and $C$ are two positive constants that depend on $\alpha$ and $\mu\,$.
Note that, if $\mu_1(\alpha)\geq \mu\,$, then $u_\alpha\equiv 0\,$.


\section{On the local energy of minimizers}\label{sec:ub-loc}

In this section and the rest of this paper, we assume that Assumption~\ref{ass:corners} holds, that $\Lambda_1$ is given by \eqref{eq:Lambda1}, and that $\Lambda_1<\mu<\Theta_0$ is a given constant.

For  any open set $D\subset\Omega$ and for any $(g,\mathbf a)\in\mathcal H$, we define the  the energy of $(g, \mathbf a)$ in $D$ by
\begin{equation}\label{eq:loc-en}
\mathcal E_0(f,\mathbf a;D)=\int_D\left(|(\nabla-i\kappa H\, \mathbf a )g|^2-\kappa^2|g|^2+\frac{\kappa^2}2|g|^4\right)\,dx\,.
\end{equation}

In the next proposition, we estimate the   energy  of $(\psi, \Ab)$ in $\mathcal N(\mathsf s,\ell)$, where $\mathsf s$ is a vertex of the domain $\Omega$,  by comparison with the reference energy $E(\mu,\alpha)$ introduced in \eqref{eq:ref-en*}. This result is new and not given in \cite{BF, CG} (especially the upper bound part).

\begin{prop}\label{prop:Ka-SIMA}
Let $\delta\in(\frac45,1)$. There exist positive constants $C\,$, $R_0\,$, and $\kappa_0$ such that for $\kappa\geq\kappa_0\,$, $H=\mu^{-1}\kappa\,$,
  the following inequalities hold with $\ell=\kappa^{-\delta}$
\begin{equation}
\begin{array}{lll}
-C\kappa^{2-3\delta}& \leq \mathcal E_0\Big( \psi,\Ab; \mathcal N\big(\mathsf s,\ell \big)\Big)- 
E(\mu,\alpha)&\leq C\,\kappa^{4-5\delta}\,,\medskip\\
 -C\,\kappa^{4-5\delta}& \leq \displaystyle\kappa^2\int_{\mathcal N(\mathsf s,\ell)}|\psi|^4\,dx+2E(\mu,\alpha)&\leq C\kappa^{2-3\delta}\,,
 \end{array}
\end{equation}
where $\mathsf s=\mathsf s_k\in\Sigma$ is a vertex of $\Omega$ (see \eqref{eq:Sigma}), $\alpha=\alpha_k\,$, $\Fb$ is introduced in \eqref{eq:F}  
and $(\psi,\Ab)_{\kappa,H}$ is a minimizer of the Ginzburg-Landau  functional introduced  in \eqref{eq-3D-GLf}.
\end{prop}
\begin{proof}~

{\bf Step~1: A useful identity.}

We will prove Proposition~\ref{prop:Ka-SIMA} by establishing matching lower and upper bounds independently.  In the sequel,
$$\gamma=\kappa^{-1/2}\,, \, \ell=\kappa^{-\delta}\,, \, \hat{\ell}=(1-\gamma)\ell\,.$$ 
 The parameter $\kappa$ is sufficiently large so that $\gamma \in (0,1)$ and  for two distinct vertices $\mathsf s\not=\mathsf s'$ of $\Omega$, $$\mathcal N(\mathsf s,\hat \ell)\cap\mathcal N(\mathsf s',\hat \ell)=\emptyset\,.$$
Pick a smooth function  $f= f_{\mathsf s,\kappa}$  satisfying, for some constant $C>0$  independent of $\kappa$,
\begin{equation}\label{eq:f}
\begin{aligned}
&f=1\ {\rm in}\ \mathcal N(\mathsf s,\hat \ell)\,,\quad f=0 \ {\rm in}\  \mathcal N\big(\mathsf s,\ell\,\big)^\complement\,,\\
& 0\leq f \leq 1,\quad |\nabla f| \leq C\gamma^{-1} \ell^{-1}\ {\rm and}\  |D^2 f| \leq C\, \gamma^{-2} \ell^{-2}\ {\rm in}\ \Omega\,.
\end{aligned}
\end{equation}
Then we can write  the following simple identity using integration by parts  (see~\cite[p.~2871]{KN})
	\begin{equation}\label{eq:Delta}
\int_{\mathcal N(\mathsf s,{\ell})}\big|(\nabla-i \kappa H {\bf A})f\psi\big|^2\,dx =\int_{\mathcal N(\mathsf s,{\ell})}\big|f(\nabla-i \kappa H {\bf A})\psi\big|^2\,dx-\int_{\mathcal N(\mathsf s,{\ell})}\,f\Delta f\, |\psi|^2\,dx\,.
\end{equation}


{\bf Step~2: Energy lower bound.}

According to  Proposition~\ref{prop:BonF}, the $L^2$-norm of $\psi$ is exponentially small as $\kappa \rightarrow +\infty$ in the support of $ \Delta  f$. Thus
\begin{align*}
\int_{\mathcal N(\mathsf s,{\ell})}\big|(\nabla-i \kappa H {\bf A})f\psi\big|^2\,dx& \leq \int_{\mathcal N(\mathsf s,{\ell})}\big|f(\nabla-i \kappa H {\bf A})\psi\big|^2\,dx+\mathcal O(\kappa^{-4})\\
&\leq\int_{\mathcal N(\mathsf s,{\ell})}\big|(\nabla-i \kappa H {\bf A})\psi\big|^2\,dx+\mathcal O(\kappa^{-4})\,.
\end{align*}
Also,  since ${\rm supp}\,(f-1)\subset \mathcal N(\mathsf s,\hat\ell)^\complement$ and $\psi$ is exponentially small there, we obtain
\begin{equation}\label{eq:1-f2}
\begin{aligned}
\int_{\mathcal N(\mathsf s, {\ell})}f^2|\psi|^2\,dx
 &=\int_{\mathcal N(\mathsf s,{\ell})}|\psi|^2\,dx -\int_{\mathcal N(\mathsf s,{\ell})}(1-f^2)|\psi|^2\,dx\\
 &=\int_{\mathcal N(\mathsf s,{\ell})}|\psi|^2\,dx +\mathcal O(\kappa^{-4})\,.
\end{aligned}
\end{equation}
Now,  we come back to the  energy in $\mathcal N(\mathsf s,\ell)$ and write
$$\mathcal E_0\Big(\psi,\Ab; \mathcal N(\mathsf s,\ell)\Big)\geq  \mathcal E_0\Big(f\psi,\Ab; \mathcal N(\mathsf s,\ell)\Big)+\mathcal O(\kappa^{-2})\,.$$
In \cite[Eq.~(6.11)]{BF}, the authors prove the following  lower bound for the energy:
$$\mathcal E_0\Big(f\psi,\Ab; \mathcal N(\mathsf s,\ell)\Big)\geq E(\mu,\alpha)+\mathcal O(\kappa^{2-3\delta})\,.$$
Consequently, we obtain the lower bound part in Proposition~\ref{prop:Ka-SIMA}:
\begin{equation}\label{eq:lb-en-corner}
\mathcal E_0\Big(\psi,\Ab; \mathcal N(\mathsf s,\ell)\Big)\geq  E(\mu,\alpha)+\mathcal O(\kappa^{2-3\delta})\,.
\end{equation}

{\bf Step~3: Energy upper bound.}

We define the function $w\in H^1(\Omega)$ as follows (see \cite{SS02})
$$w(x)=f(x) e^{-i\kappa H\eta(\Phi_{\mathsf s}(x))}\,u_\alpha\big(\sqrt{\kappa H}\,\Phi_{\mathsf s}(x)\big)+\mathbf 1_{\mathcal N(\mathsf s,\ell)^\complement}(x)\psi(x)\,,$$
where $f$ is the function introduced in \eqref{eq:f}, $\eta$ is the function introduced  in \eqref{eq:F-eta} and $\Phi_{\mathsf s}$ is the diffeomorphism that transforms $\mathcal N(\mathsf s,\ell)$ to $\Omega_\alpha\cap\{|y|<\ell\}$. Note that $w$ is in $H^1(\Omega)$ because $f=0$ on
$\mathcal N(\mathsf s,\ell)^\complement$. Finally, the function $u_\alpha$ is the minimizer of the reference functional  in \eqref{eq:ref-en}. 

Since $(\psi,\Ab)$ is a minimizer of the functional in \eqref{eq-3D-GLf}, 
$$ \mathcal E_{\rm GL}(\psi,\Ab)\leq \mathcal E_{\rm GL}(w,\Ab)\,. $$
And after dropping the magnetic energy term $\kappa^2H^2\|\curl\Ab-1\|_2^2$, we obtain
\begin{align*}
\mathcal E_0(\psi,\Ab)&\leq \mathcal E_0(w,\Ab)\\
&=\mathcal E_0(w,\Ab;\mathcal N(\mathsf s,\ell))+\mathcal E_0(\psi,\Ab;\Omega\setminus\overline{\mathcal N(\mathsf s,\ell)}\,)\,. 
\end{align*}
The obvious decomposition $$\mathcal E_0(\psi,\Ab)=\mathcal E_0(\psi,\Ab;\mathcal N(\mathsf s,\ell))+\mathcal E_0(\psi,\Ab;\Omega\setminus\overline{\mathcal N(\mathsf s,\ell)}\,)$$
 in turn yields
$$\mathcal E_0 (\psi,\Ab;\mathcal N(\mathsf s,\ell))\leq \mathcal E_0(w,\Ab;\mathcal N(\mathsf s,\ell))\,.$$
The energy $\mathcal E(w,\Ab;\mathcal N(\mathsf s,\ell))$ was estimated in \cite[Eq.~(6.4)]{BF}. Indeed, the identity in \eqref{eq:Delta}, the decay of the minimizer $u_\alpha$ in \eqref{eq:ref-minimizer} and the formula in \eqref{eq:F-eta}  yield,
$$\mathcal E_0(w,\Ab;\mathcal N(\mathsf s,\ell))\leq (1+\kappa^{-\delta})E(\mu,\alpha)+\mathcal O(\kappa^{4-5\delta})\,.$$

{\bf Step~4: Estimating the order parameter.}

Multiplying the first equation in \eqref{eq:GL} by $f^2\overline{\psi}$\,, and then integrating by parts over $\mathcal N(\mathsf s,\ell)$ we obtain the following identity,
\begin{equation}\label{eq:f'}
\mathcal E_0(f\psi,\Ab;\mathcal N(\mathsf s,\ell))=\kappa^2\int_{\mathcal N(\mathsf s,\ell)}f^2\left(-1+\frac12f^2\right)|\psi|^4\,dx+\int_{\mathcal N(\mathsf s,\ell)}|\nabla f|^2|\psi|^2\,dx\,.\end{equation}
Using the exponential decay of $\psi$ in $\mathcal N(\mathsf s,\ell)\setminus \mathcal N(\mathsf s,\hat\ell)$ (see Proposition~\ref{prop:BonF}) and the properties of the function $f$ in \eqref{eq:f}, we get
$$\mathcal E_0(f\psi,\Ab;\mathcal N(\mathsf s,\ell))=\mathcal E_0(\psi,\Ab;\mathcal N(\mathsf s,\ell))+\mathcal O(\kappa^{-4})\,,$$
and
$$\int_{\mathcal N(\mathsf s,\ell)}\left(-1+\frac12f^2\right)|\psi|^4\,dx+\int_{\mathcal N(\mathsf s,\ell)}|\nabla f|^2|\psi|^2\,dx=-\frac12\int_{\mathcal N(\mathsf s,\ell)}|\psi|^4\,dx+\mathcal O(\kappa^{-4})\,.$$
Inserting these two estimates into \eqref{eq:f'} and  then using the energy estimate for $\mathcal E_0(f\psi,\Ab;\mathcal N(\mathsf s,\ell))$, we deduce the estimate on the integral of $|\psi|^4$ over $\mathcal N(\mathsf s,\ell)$.
\end{proof}


\section{Proof of Theorem~\ref{thm:HK}}\label{sec:proof}

We use the idea introduced in our earlier paper \cite{HK} which  has some similarities with the analysis of diamagnetism \cite{FH-aif} and the computation of the quantum supercurrent  \cite{F-sc}.

Let $\delta\in(\frac45,1)$ and $\mu\in(0,\Theta_0)$ be fixed. For $-1<\epsilon<1$ and $(\psi,\Ab)_{\kappa,H=\mu^{-1}\kappa}$ a minimizer of the GL functional \eqref{eq-3D-GLf}, we define
$$\mathcal E_\epsilon(\psi,\Ab;\mathcal N(\mathsf s,\ell))=
\int_{\mathcal N(\mathsf s,\ell)} \left(|(\nabla-i\kappa H\Ab)\psi|^2-(1+\epsilon)|\psi|^2+(1+\epsilon)\kappa^2|\psi|^4\right)\,dx\,,$$
where $\ell:=\kappa^{-\delta}$ and $\mathsf s$ is a vertex of the domain $\Omega$ (with angle $\alpha$).

Let $f$ be the function satisfying \eqref{eq:f}. Using the exponential decay of $\psi$ in Proposition~\ref{eq:BonF}, 
$$ \mathcal E_\epsilon(\psi,\Ab;\mathcal N(\mathsf s,\ell))=\mathcal E_\epsilon(f\psi,\Ab;\mathcal N(\mathsf s,\ell))+\mathcal O(\kappa^{-4})$$
uniformly with respect to $\epsilon\in(-1,1)$. 

Now assume in addition that $\mu+\epsilon<\Theta_0$ (which holds when $\epsilon$ is small enough for example). We can bound the energy $\mathcal E_\epsilon(f\psi,\Ab;\mathcal N(\mathsf s,\ell))$ from  below as done in \cite[Eq.~(6.11)]{BF} (by converting the functional to a functional defined in an infinite sector  and then comparing the resulting functional with the reference energy in \eqref{eq:ref-en}). We obtain
$$\mathcal E_\epsilon(f\psi,\Ab;\mathcal N(\mathsf s,\ell))\geq E(\mu+\epsilon,\alpha)+\mathcal O(\kappa^{2-3\delta})$$
uniformly with respect to $\epsilon\in(-1,1)$.
Thus,
\begin{align*}
-\epsilon\kappa^2\int_{\mathcal N(\mathsf s,\ell)}|\psi|^2\,dx&=\mathcal E_\epsilon(\psi,\Ab;\mathcal N(\mathsf s,\ell))-\mathcal E_0(\psi,\Ab;\mathcal N(\mathsf s,\ell))-\epsilon\,\frac{\kappa^2}2\int_{\mathcal N(\mathsf s,\ell)}|\psi|^4\,dx\\
&\geq E(\mu+\epsilon,\alpha)-\mathcal E_0(\psi,\Ab;\mathcal N(\mathsf s,\ell))-\epsilon\,\frac{\kappa^2}2\int_{\mathcal N(\mathsf s,\ell)}|\psi|^4\,dx+\mathcal O(\kappa^{2-3\delta})\,,
\end{align*}
uniformly with respect to $\epsilon\in(-1,1)$. Using Proposition~\ref{prop:Ka-SIMA}, we get further
$$
-\epsilon\kappa^2\int_{\mathcal N(\mathsf s,\ell)}|\psi|^2\,dx\geq E(\mu+\epsilon,\alpha)-E(\mu,\alpha)-\epsilon\left(-E(\mu,\alpha)+\mathcal O(\kappa^{2-3\delta})\right)+\mathcal O(\kappa^{2-3\delta})\,.
$$
Divide by $\epsilon\not=0$  and then send $\kappa$ to $+ \infty$. For $\epsilon>0$, we obtain (observing that $2-3 \delta <0$)
$$\liminf_{\kappa\to+\infty}\left(-\kappa^2\int_{\mathcal N(\mathsf s,\ell)}|\psi|^2\,dx\right)\geq  \frac{E(\mu+\epsilon,\alpha)-E(\mu,\alpha)}\epsilon+E(\mu,\alpha)\,,$$
while for $\epsilon<0$,
$$\limsup_{\kappa\to+\infty}\left(-  \kappa^2 \int_{\mathcal N(\mathsf s,\ell)}|\psi|^2\,dx\right)\leq \frac{E(\mu+\epsilon,\alpha)-E(\mu,\alpha)}{\epsilon}+E(\mu,\alpha)\,.$$
Taking $\epsilon\to0_\pm$, we get the following two inequalities:
$$\liminf_{\kappa\to+\infty}\left(-  \kappa^2 \int_{\mathcal N(\mathsf s,\ell)}|\psi|^2\,dx\right)\geq E'(\mu_+,\alpha)+E(\mu,\alpha)\,,$$
and
$$\limsup_{\kappa\to+\infty}\left(- \kappa^2\int_{\mathcal N(\mathsf s,\ell)}|\psi|^2\,dx\right)\leq E'(\mu_-,\alpha)+E(\mu,\alpha)\,.$$ 
 This finishes the proof of  the asymptotic formula for $\int_{\mathcal N(\mathsf s,\ell)}|\psi|^2\,dx$ in Theorem~\ref{thm:HK}. The formula for $\int_{\mathcal N(\mathsf s,\ell)}|\psi|^4\,dx$ is already proved in Proposition~\ref{prop:Ka-SIMA}. The formula for $\int_{\mathcal N(\mathsf s,\ell)}|(\nabla-i\kappa H)\Ab)\psi|^2\,dx$ now results from the formula for 
$\mathcal E_0(\psi,\Ab;\mathcal N(\mathsf s,\ell))$ in Proposition~\ref{prop:Ka-SIMA} and the  formulas for  $\int_{\mathcal N(\mathsf s,\ell)}|\psi|^2\,dx$ and $\int_{\mathcal N(\mathsf s,\ell)}|\psi|^4\,dx$.





\end{document}